\documentclass[12pt,reqno]{amsart} 
\usepackage{amssymb,amscd,url}

\begin{document}

\allowdisplaybreaks


\title{Arithmetic and Dynamical Degrees on Abelian Varieties}
\date{\today}
\author[Joseph H. Silverman]{Joseph H. Silverman}
\email{jhs@math.brown.edu}
\address{Mathematics Department, Box 1917
         Brown University, Providence, RI 02912 USA}
\thanks{The author's research supported by Simons Collaboration Grant
  \#241309}




\hyphenation{ca-non-i-cal semi-abel-ian}


\newtheorem{theorem}{Theorem}
\newtheorem{lemma}[theorem]{Lemma}
\newtheorem{conjecture}[theorem]{Conjecture}
\newtheorem{proposition}[theorem]{Proposition}
\newtheorem{corollary}[theorem]{Corollary}
\newtheorem{claim}[theorem]{Claim}

\theoremstyle{definition}
\newtheorem*{definition}{Definition}
\newtheorem{example}[theorem]{Example}
\newtheorem{remark}[theorem]{Remark}
\newtheorem{question}[theorem]{Question}

\theoremstyle{remark}
\newtheorem*{acknowledgement}{Acknowledgements}


\newenvironment{notation}[0]{%
  \begin{list}%
    {}%
    {\setlength{\itemindent}{0pt}
     \setlength{\labelwidth}{4\parindent}
     \setlength{\labelsep}{\parindent}
     \setlength{\leftmargin}{5\parindent}
     \setlength{\itemsep}{0pt}
     }%
   }%
  {\end{list}}

\newenvironment{parts}[0]{%
  \begin{list}{}%
    {\setlength{\itemindent}{0pt}
     \setlength{\labelwidth}{1.5\parindent}
     \setlength{\labelsep}{.5\parindent}
     \setlength{\leftmargin}{2\parindent}
     \setlength{\itemsep}{0pt}
     }%
   }%
  {\end{list}}
\newcommand{\Part}[1]{\item[\upshape#1]}

\def\Case#1#2{%
\paragraph{\textbf{\boldmath Case #1: #2.}}\hfil\break\ignorespaces}

\renewcommand{\a}{\alpha}
\renewcommand{\b}{\beta}
\newcommand{\g}{\gamma}
\renewcommand{\d}{\delta}
\newcommand{\e}{\epsilon}
\newcommand{\f}{\varphi}
\newcommand{\bfphi}{{\boldsymbol{\f}}}
\renewcommand{\l}{\lambda}
\renewcommand{\k}{\kappa}
\newcommand{\lhat}{\hat\lambda}
\newcommand{\m}{\mu}
\newcommand{\bfmu}{{\boldsymbol{\mu}}}
\renewcommand{\o}{\omega}
\renewcommand{\r}{\rho}
\newcommand{\rbar}{{\bar\rho}}
\newcommand{\s}{\sigma}
\newcommand{\sbar}{{\bar\sigma}}
\renewcommand{\t}{\tau}
\newcommand{\z}{\zeta}

\newcommand{\D}{\Delta}
\newcommand{\G}{\Gamma}
\newcommand{\F}{\Phi}
\renewcommand{\L}{\Lambda}

\newcommand{\ga}{{\mathfrak{a}}}
\newcommand{\gb}{{\mathfrak{b}}}
\newcommand{\gn}{{\mathfrak{n}}}
\newcommand{\gp}{{\mathfrak{p}}}
\newcommand{\gP}{{\mathfrak{P}}}
\newcommand{\gq}{{\mathfrak{q}}}

\newcommand{\Abar}{{\bar A}}
\newcommand{\Ebar}{{\bar E}}
\newcommand{\kbar}{{\bar k}}
\newcommand{\Kbar}{{\bar K}}
\newcommand{\Pbar}{{\bar P}}
\newcommand{\Sbar}{{\bar S}}
\newcommand{\Tbar}{{\bar T}}
\newcommand{\gbar}{{\bar\gamma}}
\newcommand{\lbar}{{\bar\lambda}}
\newcommand{\ybar}{{\bar y}}
\newcommand{\phibar}{{\bar\f}}

\newcommand{\Acal}{{\mathcal A}}
\newcommand{\Bcal}{{\mathcal B}}
\newcommand{\Ccal}{{\mathcal C}}
\newcommand{\Dcal}{{\mathcal D}}
\newcommand{\Ecal}{{\mathcal E}}
\newcommand{\Fcal}{{\mathcal F}}
\newcommand{\Gcal}{{\mathcal G}}
\newcommand{\Hcal}{{\mathcal H}}
\newcommand{\Ical}{{\mathcal I}}
\newcommand{\Jcal}{{\mathcal J}}
\newcommand{\Kcal}{{\mathcal K}}
\newcommand{\Lcal}{{\mathcal L}}
\newcommand{\Mcal}{{\mathcal M}}
\newcommand{\Ncal}{{\mathcal N}}
\newcommand{\Ocal}{{\mathcal O}}
\newcommand{\Pcal}{{\mathcal P}}
\newcommand{\Qcal}{{\mathcal Q}}
\newcommand{\Rcal}{{\mathcal R}}
\newcommand{\Scal}{{\mathcal S}}
\newcommand{\Tcal}{{\mathcal T}}
\newcommand{\Ucal}{{\mathcal U}}
\newcommand{\Vcal}{{\mathcal V}}
\newcommand{\Wcal}{{\mathcal W}}
\newcommand{\Xcal}{{\mathcal X}}
\newcommand{\Ycal}{{\mathcal Y}}
\newcommand{\Zcal}{{\mathcal Z}}

\renewcommand{\AA}{\mathbb{A}}
\newcommand{\BB}{\mathbb{B}}
\newcommand{\CC}{\mathbb{C}}
\newcommand{\FF}{\mathbb{F}}
\newcommand{\GG}{\mathbb{G}}
\newcommand{\NN}{\mathbb{N}}
\newcommand{\PP}{\mathbb{P}}
\newcommand{\QQ}{\mathbb{Q}}
\newcommand{\RR}{\mathbb{R}}
\newcommand{\ZZ}{\mathbb{Z}}

\newcommand{\bfa}{{\mathbf a}}
\newcommand{\bfb}{{\mathbf b}}
\newcommand{\bfc}{{\mathbf c}}
\newcommand{\bfd}{{\mathbf d}}
\newcommand{\bfe}{{\mathbf e}}
\newcommand{\bff}{{\mathbf f}}
\newcommand{\bfg}{{\mathbf g}}
\newcommand{\bfp}{{\mathbf p}}
\newcommand{\bfr}{{\mathbf r}}
\newcommand{\bfs}{{\mathbf s}}
\newcommand{\bft}{{\mathbf t}}
\newcommand{\bfu}{{\mathbf u}}
\newcommand{\bfv}{{\mathbf v}}
\newcommand{\bfw}{{\mathbf w}}
\newcommand{\bfx}{{\mathbf x}}
\newcommand{\bfy}{{\mathbf y}}
\newcommand{\bfz}{{\mathbf z}}
\newcommand{\bfA}{{\mathbf A}}
\newcommand{\bfF}{{\mathbf F}}
\newcommand{\bfB}{{\mathbf B}}
\newcommand{\bfD}{{\mathbf D}}
\newcommand{\bfG}{{\mathbf G}}
\newcommand{\bfI}{{\mathbf I}}
\newcommand{\bfM}{{\mathbf M}}
\newcommand{\bfzero}{{\boldsymbol{0}}}

\newcommand{\Aut}{\operatorname{Aut}}
\newcommand{\Base}{\mathcal{B}} 
\newcommand{\CM}{\operatorname{CM}}   
\newcommand{\codim}{\operatorname{codim}}
\newcommand{\diag}{\operatorname{diag}}
\newcommand{\Disc}{\operatorname{Disc}}
\newcommand{\Div}{\operatorname{Div}}
\newcommand{\Dom}{\operatorname{Dom}}
\newcommand{\Ell}{\operatorname{Ell}}   
\newcommand{\End}{\operatorname{End}}
\newcommand{\Fbar}{{\bar{F}}}
\newcommand{\Gal}{\operatorname{Gal}}
\newcommand{\GL}{\operatorname{GL}}
\newcommand{\Hom}{\operatorname{Hom}}
\newcommand{\hplus}{h^{\scriptscriptstyle+}}
\newcommand{\Index}{\operatorname{Index}}
\newcommand{\Image}{\operatorname{Image}}
\newcommand{\liftable}{{\textup{liftable}}}
\newcommand{\hhat}{{\hat h}}
\newcommand{\Ker}{{\operatorname{ker}}}
\newcommand{\Lift}{\operatorname{Lift}}
\newcommand{\Mat}{\operatorname{Mat}}
\newcommand{\MOD}[1]{~(\textup{mod}~#1)}
\newcommand{\Mor}{\operatorname{Mor}}
\newcommand{\Moduli}{\mathcal{M}}
\newcommand{\Norm}{{\operatorname{\mathsf{N}}}}
\newcommand{\notdivide}{\nmid}
\newcommand{\normalsubgroup}{\triangleleft}
\newcommand{\NotDom}{\operatorname{NotDom}}
\newcommand{\NS}{\operatorname{NS}}
\newcommand{\odd}{{\operatorname{odd}}}
\newcommand{\onto}{\twoheadrightarrow}
\newcommand{\ord}{\operatorname{ord}}
\newcommand{\Orbit}{\mathcal{O}}
\newcommand{\Per}{\operatorname{Per}}
\newcommand{\Perp}{\operatorname{Perp}}
\newcommand{\PrePer}{\operatorname{PrePer}}
\newcommand{\PGL}{\operatorname{PGL}}
\newcommand{\Pic}{\operatorname{Pic}}
\newcommand{\Prob}{\operatorname{Prob}}
\newcommand{\Qbar}{{\bar{\QQ}}}
\newcommand{\rank}{\operatorname{rank}}
\newcommand{\Rat}{\operatorname{Rat}}
\newcommand{\rel}{{\textup{rel}}}  
\newcommand{\Resultant}{\operatorname{Res}}
\renewcommand{\setminus}{\smallsetminus}
\newcommand{\shat}{{\hat s}}
\newcommand{\sing}{{\textup{sing}}} 
\newcommand{\sgn}{\operatorname{sgn}} 
\newcommand{\SL}{\operatorname{SL}}
\newcommand{\Span}{\operatorname{Span}}
\newcommand{\Spec}{\operatorname{Spec}}
\newcommand{\Support}{\operatorname{Supp}}
\newcommand{\That}{{\hat T}}  
\newcommand{\tors}{{\textup{tors}}}
\newcommand{\Trace}{\operatorname{Trace}}
\newcommand{\tr}{{\textup{tr}}} 
\newcommand{\UHP}{{\mathfrak{h}}}    
\newcommand{\<}{\langle}
\renewcommand{\>}{\rangle}

\newcommand{\ds}{\displaystyle}
\newcommand{\longhookrightarrow}{\lhook\joinrel\longrightarrow}
\newcommand{\longonto}{\relbar\joinrel\twoheadrightarrow}

\newcount\ccount \ccount=0
\def\cc{\global\advance\ccount by1{c_{\the\ccount}}}


\begin{abstract}
Let \(\phi:X\dashrightarrow X\) be a dominant rational map of a smooth
variety and let \(x\in X\), all defined over \(\bar{\mathbb Q}\).  The
dynamical degree \(\delta(\phi)\) measures the geometric complexity of
the iterates of \(\phi\), and the arithmetic degree \(\alpha(\phi,x)\)
measures the arithmetic complexity of the forward \(\phi\)-orbit of
\(x\).  It is known that \(\alpha(\phi,x)\le\delta(\phi)\), and it is
conjectured that if the \(\phi\)-orbit of \(x\) is Zariski dense in
\(X\), then \(\alpha(\phi,x)=\delta(\phi)\), i.e., arithmetic
complexity equals geometric complexity. In this note we prove this
conjecture in the case that \(X\) is an abelian variety, extending
earlier work in which the conjecture was proven for isogenies.
\end{abstract}


\maketitle


\section{Introduction}
Let $K=\Qbar$, or more generally let $K$ be an algebraically closed
field on which one has a good theory of height functions as described,
for example, in \cite[Part~B]{hindrysilverman:diophantinegeometry} or
\cite[Chapters~2--5]{lang:diophantinegeometry}. Let~$X$ be a smooth
projective variety of dimension~$d$, let $\f:X\dashrightarrow X$ be a
dominant rational map, and let~$H$ be an ample divisor on~$X$, all
defined over~$K$. Further let $h_{X,H}:X(K)\to[1,\infty)$ be a Weil
height function associated to~$H$.  We write~$\f^n$ for the $n$'th
iterate of~$\f$.

\begin{definition}
The \emph{dynanmical degree of~$\f$} is the quantity
\[
  \d(\f) = \lim_{n\to\infty} 
     \Bigl(\bigl((\f^n)^*H\bigr) \cdot H^{d-1}\Bigr)^{1/n}.
\]
\end{definition}

\begin{definition}
Let $x\in X$ be a point whose forward $\f$-orbit
\[
  \Orbit_\f(x)=\{\f^n(x):n\ge0\}
\]
is well-defined.  The \emph{arithmetic degree of~$x$} (\emph{relative
  to~$\f$}) is the quantity
\[
  \a(\f,x) = \lim_{n\to\infty} h_{X,H}\bigl(f^n(x)\bigr)^{1/n}.
\]
\end{definition}
 
It is known that the limit defining~$\d(\f)$ exists and is a
birational invariant; see~\cite[Proposition~1.2(iii)]{MR2179389}
and~\cite[Corollary~16]{kawsilv:arithdegledyndeg}.  Bellon and
Viallet~\cite{MR1704282} conjectured that~$\d(\f)$ is an algebraic
integer.  Kawaguchi and the author~\cite{kawsilv:arithdegledyndeg}
proved that~$\a(\f,x)\le\d(\f)$, and they made the following
conjectures about the arithmetic degree and its relation to the
dynamical degree.

\begin{conjecture}
\label{conjecture:mainconj}
\textup{(Kawaguchi--Silverman \cite{kawsilv:arithdegledyndeg,arxiv1111.5664})}
\begin{parts}
\Part{(a)}
The limit defining $\a(\f,x)$ exists.
\Part{(b)}
$\a(\f,x)$ is an algebraic integer.
\Part{(c)}
$\bigl\{\a(\f,x) : \text{$x\in X$ such that $\Orbit_\f(x)$ exists}\bigr\}$
is a finite set.
\Part{(d)}
If the orbit $\Orbit_\f(x)$ is Zariski dense in~$X$, then
$\a(\f,x)=\d(\f)$.
\end{parts}
\end{conjecture}

Conjecture~\ref{conjecture:mainconj}(a,b,c) is known when~$\f$ is a
morphism~\cite{kawsilv:jordanblock}.
But~(d), which is the deepest part of the conjecture, has been proven
in only a few situations, such as group endomorphisms of the
torus~$\GG_m^d$ \cite[Theorem~4]{arxiv1111.5664}, group endomorphisms
of abelian varieties~\cite[Theorem~4]{kawsilv:jordanblock}, and in a few other
special cases; see~\cite[Section~8]{kawsilv:arithdegledyndeg}. The
goal of this note is to extend the result
in~\cite{kawsilv:jordanblock} to arbitrary dominant self-maps of abelian
varieties.

\begin{theorem}
\label{theorem:mainthm}
Let~$A/K$ be an abelian variety, let~$\f:A\to A$ be a dominant
rational map, and let~$P\in A$ be a point whose orbit~$\Orbit_\f(P)$
is Zariski dense in~$A$. Then
\[
  \a(\f,P)=\d(\f).
\]
\end{theorem}

\begin{remark}
Every map as in Theorem~\ref{theorem:mainthm} is a composition of a
translation and an isogeny (see Remark~\ref{remark:rigidity}), so we
can write~$\f:A\to A$ as
\[
  \f(P) = f(P) + Q
\]
with~$f:A\to A$ an isogeny and~$Q\in A$.  As already noted, if~$Q=0$,
then Theorem~\ref{theorem:mainthm} was proven
in~\cite{kawsilv:jordanblock}, and it may seem that the introduction
of translation by a non-zero~$Q$ introduces only a minor complication
to the problem. However, the potential interaction between the
points~$P$ and~$Q$ may lead to signficiant changes in both the orbit of~$P$
and the value of the arithmetic degree~$\a(\f,P)$ .

To illustrate the extent to which taking~$Q\ne0$ is significant,
consider the following related question. For which~$\f$ are there any
points~$P\in A$ whose $\f$-orbit~$\Orbit_\f(P)$ is Zariski dense
in~$A$?  If~$Q=0$, this question is easy to answer, e.g., by using
Poincar\'e reducibility~\cite[Section~19, Theorem~1]{MR0282985}. But
if~$Q\ne0$ and the field~$K$ is countable, for example~$K=\Qbar$, then
the problem becomes considerably more difficult. Indeed, the solution,
which only recently appeared in~\cite{arxiv1412.2029}, uses Faltings'
theorem (Mordell--Lang conjecture) on the intersection of subvarieties
of~$A$ with finitely generated subgroups of~$A$. So at present it
requires deep tools to determine whether there exist any points~$P\in
A(K)$ to which Theorem~\ref{theorem:mainthm} applies.
\end{remark}

We briefly outline the contents of this note. We begin in
Section~\ref{section:notation} by setting notation.
Section~\ref{section:prelim} contains a number of preliminary results
describing how dynamical and arithmetic degrees vary in certain
situations. We then apply these tools and results from earlier work to
give the proof of Theorem~\ref{theorem:mainthm} in
Section~\ref{section:proofmainthm}. Finally, in
Section~\ref{section:auxlemma} we prove an auxiliary lemma on
pullbacks and pushforwards of divisors that is needed for one of the
proofs in Section~\ref{section:prelim}.

\section{Notation}
\label{section:notation}

We set the following notation, which will be used for the remainder
of this note.

\begin{notation}
\item[$K$]
an algebraically closed field on which their is a good theory of height
functions. For example, the algebraic closure of~$\QQ$ or 
the algebraic closure of a one-dimensional function field.
\item[$A/K$]
an abelian variety of dimension $d$ defined over $K$.
\item[$Q$]
a point in $A(K)$.
\item[$\t_Q$]
the translation-by-$Q$ map,
\[
  \t_Q:A\longrightarrow A,\quad
  \t_Q(P)=P+Q.
\]
\item[$f$]
an isogeny $f:A\to A$ defined over~$K$.
\item[$\f$]
the finite map $\f:A\to A$ given by $\f=\t_Q\circ f$, i.e.,
\[
  \f(P)=f(P)+Q.
\]
\item[$h_{A,H}$] 
a height function $h_{A,H}:A(K)\to\RR$ associated to an ample
divisor~$H\in\Div(A)$; see for
example~\cite{hindrysilverman:diophantinegeometry,lang:diophantinegeometry}.
\item[$\f^n$]
the $n$'th iterate of~$\f$, i.e., $\f^n(P)=\f\circ\f\circ\cdots\circ\f(P)$.
\item[$\Orbit_\f(P)$]
the forward $\f$-orbit of~$P$, i.e., the set $\{\f^n(P):n\ge0\}$.
\end{notation}

\begin{remark}
\label{remark:rigidity}
It is a standard fact that every rational map~$A\dashrightarrow A$ is
a morphism, and that every finite morphism $A\to A$ is the composition
of an isogeny and a translation~\cite[Section~4,
  Corollary~1]{MR0282985}.  Hence the set of dominant rational
maps~$A\dashrightarrow A$ is the same as the set of maps of the
form~$\f=\t_Q\circ f$ as in our notation.
\end{remark}

As noted earlier, since~$\f:A\to A$ is a morphism, it is
known~\cite{kawsilv:arithdegledyndeg} that the limit
defining~$\a_\f(P)$ exists (and is an algebraic integer).

\section{Preliminary material}
\label{section:prelim}
In this section we collect some basic results that are needed to prove
Theorem~\ref{theorem:mainthm}. We begin with a standard (undoubtedly
well-known) decomposition theorem.

\begin{lemma}
\label{lemma:AB1B2}
Let $A$ be an abelian variety, let~$f:A\to A$ be an isogeny, and
let~$F(X)\in\ZZ[X]$ be a polynomial such that~$F(f)=0$ in
$\End(A)$. Suppose that~$F$ factors as
\[
  F(X)=F_1(X)F_2(X)\quad\text{with $F_1,F_2\in\ZZ[X]$ and
    $\gcd(F_1,F_2)=1$,}
\]
where the gcd is computed in~${\QQ[X]}$.  Let
\[
  A_1 = F_1(f)A\quad\text{and}\quad A_2=F_2(f)A,
\]
so~$A_1$ and~$A_2$ are abelian subvarieties of~$A$. Then we have\textup:
\begin{parts}
\Part{(a)}
$A=A_1+A_2$.
\Part{(b)}
$A_1\cap A_2$ is finite.
\end{parts}
More precisely, if we let $\rho=\Resultant(F_1,F_2)$,
then $A_1\cap A_2\subset A[\rho]$.
\end{lemma}
\begin{proof}
The gcd assumption on~$F_1$ and~$F_2$ implies that their resultant is
non-zero, so we can find polynomials~$G_1,G_2\in\ZZ[X]$ so that
\[
  G_1(X)F_1(X)+G_2(X)F_2(X) = \rho = \Resultant(F_1,F_2) \ne 0.
\]
We observe that $fA_1\subset A_1$ and $fA_2\subset A_2$ and compute
\begin{align*}
  A = \rho A
  &= \bigl(G_1(f)F_1(f)+G_2(f)F_2(f)\bigr)A \\*
  &= G_1(f)A_1 + G_2(f)A_2 \\*
  &\subset A_1+A_2 \subset A.
\end{align*}
Hence $A=A_1+A_2$. This proves~(a). For~(b), suppose that
$P\in A_1\cap A_2$, so
\[
  P = F_1(f)P_1  = F_2(f)P_2
  \quad\text{for some $P_1\in A_1$ and $P_2\in A_2$.}
\]
Then
\begin{align*}
  \rho P
  &= \bigl(G_1(f)F_1(f)+G_2(f)F_2(f)\bigr)P \\
  &= G_1(f)F_1(f)F_2(f)P_2 + G_2(f)F_2(f)F_1(f)P_1\\
  &= G_1(f)F(f)P_2 + G_2(f)F(f)P_1 
   \quad\text{since $F=F_1F_2$,}\\
  &= 0\quad\text{since $F(f)=0$.}
\end{align*}
Hence $A_1\cap A_2\subset A[\rho]$.
\end{proof}

The next two lemmas relate dynamical and arithmetic degrees.  We state
them somewhat more generally than needed in this note, since the proofs are
little more difficult and they may be useful for future applications.
The first lemma says that dynamical and arithmetic degrees are
invariant under finite maps, and the second describes dynamical and
arithmetic degrees on products. 

\begin{lemma}
\label{lemma:finitemapsameda}
Let $X$ and $Y$ be non-singular projective varieties, and let
\[
  \begin{CD}
   X @>\l>> Y \\
  @VV f_X V @VV f_Y V \\
   X @>\l>> Y \\
  \end{CD}
\]
be a commutative diagram, where~$f_X$ and~$f_Y$ are dominant rational
maps and~$\l$ is a finite map, with everything defined over~$K$
\begin{parts}
\Part{(a)}
Let~$x\in X$ whose orbit~$\Orbit_{f_X}(x)$ is
well-defined. Then~$\Ocal_{f_X}(x)$ is Zariski dense in~$X$ if and
only if~$\Ocal_{f_Y}\bigl(\l(x)\bigr)$ is Zariski dense in~$Y$.
\Part{(b)}
The dynamical degrees of~$f_X$ and $f_Y$ are equal,
\[
  \d(f_X)=\d(f_Y)
\]
\Part{(c)}
Let~$P\in X$ be a point such that the forward~$f_X$-orbit of~$P$
and the arithmetic degree of~$P$ relative to~$f_X$ are
well-defined. Then the arithmetic degrees of~$P$ and~$\l(P)$ satisfy
\[
  \a(f_X,P) = \a\bigl(f_Y,\l(P)\bigr).
\]
\end{parts}
\end{lemma}

\begin{remark}
Lemma~\ref{lemma:finitemapsameda} is a special (relatively easy) case
of results of Dinh--Nguyen~\cite{MR2851870} and
Dinh--Nguyen--Truong~\cite{MR2989646}.  For completeness, we give an
algebraic proof, in the spirit of the present paper, which works in
arbitrary characteristic.
\end{remark}

\begin{proof}[Proof of Lemma $\ref{lemma:finitemapsameda}$]
(a) 
We first remark that the~$f_Y$ orbit of~$\l(x)$ is also well-defined.
To see this, let~$n\ge1$ and let~$U$ be any Zariski open set on
which~$f_X^n$ is well-defined. Then~$\l\circ f_X^n$ is also
well-defined on~$U$, since~$\l$ is a morphism. Also, since~$\l$ is a
finite map, the image~$\l(U)$ is a Zariski open set, and we note
that~$f_Y^n$ on the set~$\l(U)$ agrees with~$\l\circ f_X^n$
on~$U$. Thus~$f_Y^n$ is defined on~$\l(U)$. In particular,
since~$f_X^n$ is assumed defined at~$x$, we see that~$f_Y^n$ is
defined at~$\l(x)$.

Suppose that $Z=\overline{\Ocal_{f_X}(x)}\ne
X$. Then~$\l(Z)$ is a proper Zariski closed subset of~$Y$, since
finite maps send closed sets to closed sets. Further,
\[
  \Ocal_{f_Y}\bigl(\l(x)\bigr) = \l\bigl(\Ocal_{f_X}(x)\bigr)
  \subset \l(Z).
\]
Hence $\Ocal_{f_Y}\bigl(\l(x)\bigr)$ is not Zariski dense.
Conversely, suppose that $W=\overline{\Ocal_{f_Y}\bigl(\l(x)\bigr)}\ne
Y$. Finite maps (and indeed, morphisms) are continuous for the Zariski
topology, so~$\l^{-1}(W)$ is a closed subset of~$X$, and the fact
that~$\l$ is a finite map, hence surjective, implies
that~$\l^{-1}(W)\ne X$. Then
\[
  \Ocal_{f_X}(x) 
  \subset \l^{-1}\Bigl(\Ocal_{f_Y}\bigl(\l(x)\bigr)\Bigr)
  \subset \l^{-1}(W) \subsetneq X,
\]
so $\Ocal_{f_X}(x)$ is not Zariski dense in~$X$.
\par\noindent(b)\enspace
Let $d=\dim(X)=\dim(Y)$, and let $H_Y$ be an ample divisor on~$Y$. The
assumption that~$\l$ is a finite morphism implies that~$H_X:=\l^*H_Y$
is an ample divisor on~$X$.  This follows from
\cite[Exercise~5.7(d)]{hartshorne}, or we can use the Nakai--Moishezon
Criterion \cite[Theorem~A.5.1]{hartshorne} and note that for every
irreducible subvariety~$W\subset X$ of dimension~$r$ we have
\[
  \l_*(H_X\cdot W^r) = \l_*(\l^*H_Y\cdot W^r) = H_Y\cdot(\l_*W)^r > 0,
\]
since~$\l_*W$ is a positive multiple of an $r$-dimensional irreducible
subvariety of~$Y$.  This means that we can use~$H_X$ to
compute~$\d(f_X)$. In the following computation we use
that fact that since~$\l$ is a finite morphism, we have
\begin{equation}
  \label{eqn:fxNllfYN}
  (f_X^N)^*\circ\l^*
  = (\l\circ f_X^N)^*
  = (f_Y^N\circ\l)^*
  = \l^*\circ (f_Y^N)^*.
\end{equation}
We give the justification for this formula at the end of this paper
(Lemma~\ref{lemma:lffl}), but we note that for the proof of
Theorem~\ref{theorem:mainthm}, all of the relevant maps are morphisms,
so~\eqref{eqn:fxNllfYN} is trivially true. We compute
\begin{align*}
  \d(f_X)
  &= \lim_{n\to\infty} \Bigl( (f_X^n)^*H_X \cdot H_X^{d-1} \Bigr)^{1/n} \\*
  &= \lim_{n\to\infty} \Bigl( (f_X^n)^*\circ\l^*H_Y \cdot (\l^*H_Y)^{d-1}
           \Bigr)^{1/n} \\
  &= \lim_{n\to\infty} \Bigl( \l^*\circ(f_Y^n)^*H_Y \cdot (\l^*H_Y)^{d-1}
           \Bigr)^{1/n}
     \quad\text{from \eqref{eqn:fxNllfYN},} \\
  &= \lim_{n\to\infty} \Bigl( \deg(\l) \bigl((f_Y^n)^*H_Y \cdot ^*H_Y^{d-1}\bigr)
           \Bigr)^{1/n} \\
  &= \lim_{n\to\infty} \Bigl( (f_Y^n)^*H_Y \cdot H_Y^{d-1} \Bigr)^{1/n} \\*
  &= \d(f_Y).
\end{align*}
This completes the proof of~(b).
\par\noindent(c)\enspace
We do an analogous height computation, where the~$O(1)$ quantities
depend on~$X$,~$Y$,~$\l$,~$f_X$,~$f_Y$, and the choice of height functions
for~$H_X$ and~$H_Y$, but do not depend of~$n$.
\begin{align*}
  \a(f_X,P)
  &= \lim_{n\to\infty} h_{X,H_X}\bigl(f_X^n(P)\bigr)^{1/n} \\*
  &= \lim_{n\to\infty} h_{X,\l^*H_Y}\bigl(f_X^n(P)\bigr)^{1/n} \\
  &= \lim_{n\to\infty} \Bigl(
   h_{X,H_Y}\bigl(\l\circ f_X^n(P)\bigr) + O(1) \Bigr)^{1/n} \\
  &= \lim_{n\to\infty} \Bigl(
   h_{X,H_Y}\bigl(f_Y^n\circ\l(P)\bigr) + O(1) \Bigr)^{1/n} \\*
  &= \a\bigl(f_Y,\l(P)\bigr).
\end{align*}
This completes the proof of~(c).
\end{proof}

\begin{lemma}
\label{lemma:ddadproduct}
Let $Y$ and $Z$ be non-singular projective varieties,  let
\[
  f_Y:Y\to Y\quad\text{and}\quad
  f_Z:Z\to Z
\]
be dominant rational maps, and let $f_{Y,Z}:=f_Y\times f_Z$ be the
induced map on the product~$Y\times Z$, with everything defined over~$K$.  
\begin{parts}
\Part{(a)}
Let~$y\in Y$ and~$z\in Z$ be points whose forward orbits via~$f_Y$,
respectively~$f_Z$, are well-defined, and suppose that
$\Orbit_{f_{Y,Z}}(y,z)$ is Zariski dense in $Y\times Z$.  Then
$\Orbit_{f_{Y}}(y)$ is Zariski dense in $Y$ and $\Orbit_{f_{Z}}(z)$ is
Zariski dense in~$Z$.
\Part{(b)}
The dynamical degrees of~$f_Y$,~$f_Z$, and~$f_{Y,Z}$ are related by
\[
  \d(f_{Y,Z})=\max\bigl\{\d(f_Y),\d(f_Z)\bigr\}.
\]
\Part{(c)}
Let~$(P_Y,P_Z)\in (Y\times Z)(K)$ be a point such that the arithmetic
degrees~$\a(f_Y,P_Y)$ and~$\a(f_Z,P_Z)$ are well-defined. Then
\[
  \a\bigl(f_{Y,Z},(P_Y,P_Z)\bigr)
    =\max\bigl\{\a(f_Y,P_Y),\a(f_Z,P_Z)\bigr\}.
\]
\end{parts}
\end{lemma}
\begin{proof}[Proof of $\ref{lemma:ddadproduct}$]
(a )
This elementary fact has nothing to do with orbits.  Let $S\subset Y$
and $T\subset Z$ be sets of points. By symmetry, it suffices to prove
that if $S\times T$ is Zariski dense in~$Y\times Z$, then~$S$ is
Zariski dense in~$Y$. We prove the contrapositive, so assume that~$S$
is not Zariski dense in~$Y$. This means that there is a proper Zariski
closed subset~$W\subset Y$ with $S\subset W$. Then $S\times T\subset
W\times Z\subsetneq Y\times Z$, which shows that~$S\times T$ is not
Zariski dense in~$Y\times Z$.  \par\noindent(b)\enspace Let
\[
  \pi_Y:Y\times Z\to Y
  \quad\text{and}\quad
  \pi_Z:Y\times Z\to Z
\]
denote the projection maps.  Let $H_Y$ and $H_Z$ be, respectively,
ample divisors on~$Y$ and~$Z$.  Then
\[
  H_{Y,Z} := (H_Y\times Z) + (Y\times H_Z) = \pi_Y^*H_Y + \pi_Z^*H_Z
\]
is an ample divisor on~$Y\times Z$.  We compute
\begin{align*}
  (f_{Y,Z}^n)^*H_{Y,Z}
  &= (f_Y^n\times f_Z^n)^*(\pi_Y^*H_Y + \pi_Z^*H_Z) \\
  &= \pi_Y^*\circ (f_Y^n)^* H_Y + \pi_Z^*\circ (f_Z^n)^* H_Z.
\end{align*}
We let 
\[
  d_Y=\dim(Y),\quad  d_Z=\dim(Z),\quad\text{so}\quad
  \dim(Y\times Z) = d_Y+d_Z.
\]
We compute
\begin{align}
  \label{eqn:fYZnHYZ}
  (f_{Y,Z}^n)^*&H_{Y,Z}\cdot H_{Y,Z}^{d_Y+d_Z-1}\notag\\
  &= \Bigl(\pi_Y^*\circ (f_Y^n)^* H_Y + \pi_Z^*\circ (f_Z^n)^* H_Z\bigr)
       \cdot \Bigl( \pi_Y^*H_Y + \pi_Z^*H_Z \Bigr)^{d_Y+d_Z-1} \notag \\
  &=  \binom{d_y+d_Z-1}{d_Z}((f_Y^n)^* H_Y \cdot H_Y^{d_Y-1} )(H_Z^{d_Z}) \notag\\
  &\hspace{4em}
    {}  + \binom{d_y+d_Z-1}{d_Y}((f_Z^n)^* H_Z \cdot H_Z^{d_Z-1} )(H_Y^{d_Y}).
\end{align}

For any dominant rational self-map $f:X\to X$ of a non-singular
projective variety of dimension~$d$ and any ample divisor~$H$ on~$X$,
the dynamical degree of~$f$ is, by definition, the number~$\d(f)$
satisfying
\[
  \d(f)^n = (f^n)^*H \cdot H^{d-1} \cdot 2^{o(n)}
  \quad\text{as $n\to\infty$.}
\]
Using this formula three times in~\eqref{eqn:fYZnHYZ} yields
\[
  \d(f_{Y,Z})^n\cdot 2^{o(n)}
  = \d(f_Y)^n\cdot 2^{o(n)} \cdot H_Z^{d_Z}
        +  \d(f_Z)^n\cdot 2^{o(n)} \cdot H_Y^{d_Y}.
\]
The quantities $H_Y^{d_Y}$ and $H_Z^{d_Z}$ are positive, since~$H_Y$
and~$H_Z$ are ample. Now taking the $n$'th root of both sides and
letting~$n\to\infty$ gives the desired result, which completes
the proof of~(b).
\par\noindent(c)\enspace
We do a similar computation. Thus
\begin{align*}
  h&_{Y\times Z,H_{Y,Z}} \bigl(f_{Y,Z}^n(P_Y,P_Z)\bigr) \\
  &= h_{Y\times Z,\pi_Y^*H_Y} \bigl(f_{Y,Z}^n(P_Y,P_Z)\bigr)
   + h_{Y\times Z,\pi_Z^*H_Z} \bigl(f_{Y,Z}^n(P_Y,P_Z)\bigr) + O(1) \\
  &= h_{Y,H_Y} \bigl(\pi_Y\circ f_{Y,Z}^n(P_Y,P_Z)\bigr)
   + h_{Z,H_Z} \bigl(\pi_Z\circ f_{Y,Z}^n(P_Y,P_Z)\bigr) + O(1) \\
  &= h_{Y,H_Y} \bigl(f_Y^n(P_Y)\bigr)
   + h_{Z,H_Z} \bigl(f_Z^n(P_Z)\bigr) + O(1).
\end{align*}
For any dominant rational self-map $f:X\to X$ of a non-singular
projective variety defined over~$K$, any ample divisor~$H$ on~$X$,
and any~$P\in X(K)$ whose $f$-orbit is well-defined, the arithmetic
degree is the limit (if it exists)
\[
  \a(f,P) :=
  \lim_{n\to\infty} \hplus_{X,H}\bigl(f^n(P)\bigr)^{1/n}.
\]
(Here $\hplus=\max\{h,1\}$.) Hence
\begin{align*}
  \a\bigl(f_{Y,Z},(P_Y,P_Z)\bigr)
  &= \lim_{n\to\infty}
  \hplus_{Y\times Z,H_{Y,Z}} \bigl(f_{Y,Z}^n(P_Y,P_Z)\bigr)^{1/n} \\
  &= \lim_{n\to\infty}
  \Bigl( \hplus_{Y,H_Y} \bigl(f_Y^n(P_Y)\bigr)
   + \hplus_{Z,H_Z} \bigl(f_Z^n(P_Z)\bigr) + O(1) \Bigr)^{1/n} \\
  & =\max\bigl\{\a(f_Y,P_Y),\a(f_Z,P_Z)\bigr\},
\end{align*}
which completes the proof of~(c).
\end{proof}

\section{Proof of Theorem~$\ref{theorem:mainthm}$}
\label{section:proofmainthm}

\begin{proof}[Proof of Theorem~$\ref{theorem:mainthm}$]
The translation map~$\t_Q$ induces the identity map\footnote{Let
  $\mu:A\times A\to A$ be $\mu(x,y)=x+y$,  and let $D\in\Div(A)$.  Then for
  any~$P\in A$, the divisor $\mu^*D$ has the property that
  $\mu^*D|_{A\times\{P\}}=\t_P^*D\times\{P\}$. Hence as~$P$ varies, the
  divisors~$\t_P^*D$ are algebraically equivalent, so in particular
  $\t_Q^*D\equiv\t_0^*D\equiv D$, which shows that~$\t_Q^*$ is the
  identity map on~$\NS(A)$.}
\[
  \t_Q^*=\text{id}:\NS(A)\to\NS(A),
\]
from which we deduce that
\begin{equation}
  \label{eqn:dphidf}
  \f^*=f^*\quad\text{and}\quad \d_\f=\d_f.
\end{equation}

We begin by proving Theorem~\ref{theorem:mainthm} under the assumption
that a non-zero multiple of the point~$Q$ is in the image of the
map~$f-1$, say
\[
  mQ = (f-1)(Q')\quad\text{for some $m\ne0$ and $Q'\in A$.}
\]
Then we have
\begin{align}
  \label{eqn:phinPfnPQ}
  m \f^n(P)
  &= m \bigl(f^n(P) + (f^{n-1}+f^{n-2}+\cdots+f+1)(Q)\bigr) \notag \\*
  &=  f^n(mP) + (f^{n-1}+f^{n-2}+\cdots+f+1)(mQ) \notag \\*
  &= f^n(mP) + (f^{n-1}+f^{n-2}+\cdots+f+1)\circ(f-1)(Q') \notag \\
  &= f^n(mP)+f^n(Q')-Q' \notag \\*
  &= f^n(mP+Q')-Q'.
\end{align}
In particular, the~$\f$ orbit of~$P$ and the~$f$ orbit of~$mP+Q'$
differ by translation by~$-Q'$, so the assumption that~$\Orbit_\f(P)$
is Zariski dense and the fact that translation is an automorphism
imply that~$\Orbit_f(mP+Q')$ is also Zariski dense. 
We will also use the standard formula
\begin{equation}
  \label{eqn:hcircmm2h}
  h_{A,H}\circ m = m^2 h_{A,H}+O(1).
\end{equation}

We now compute (with additional explanation for steps~\eqref{eqn:afP1}
and~\eqref{eqn:afP2} following the computation)
\begin{align}
  \a_\f(P)
  &= \lim_{n\to\infty} h_{A,H}\bigl(\f^n(P)\bigr)^{1/n} 
     &&\text{by definition,} \notag \\
  &= \lim_{n\to\infty} h_{A,H}\bigl(m\f^n(P)\bigr)^{1/n} 
     &&\text{from \eqref{eqn:hcircmm2h},} \notag\\
  &= \lim_{n\to\infty} h_{A,H}\bigl(f^n(mP+Q')-Q'\bigr)^{1/n} 
     &&\text{from \eqref{eqn:phinPfnPQ},} \notag \\
  &= \lim_{n\to\infty} h_{A,\t_{-Q'}^*H}\bigl(f^n(mP+Q')\bigr)^{1/n} 
     &&\text{functoriality,} \notag \\
  &= \a_f(mP+Q')
     &&\text{by definition,} \label{eqn:afP1}\\
  &= \d_f
     &&\text{from
    \cite[Theorem 4]{kawsilv:jordanblock},} \label{eqn:afP2}\\
  &= \d_\f
     &&\text{from \eqref{eqn:dphidf}.} \notag
\end{align}
We note that~\eqref{eqn:afP1} follows
from~\cite[Proposition~12]{kawsilv:arithdegledyndeg}, which says that
the arithmetic degree may be computed using the height relative to any
ample divisor. (The map~$\t_{-Q'}$ is an isomorphism, so~$\t_{-Q'}^*H$
is ample.)  For~\eqref{eqn:afP2}, we have applied
\cite[Theorem~4]{kawsilv:jordanblock} to the isogeny~$f$ and the
point~$mP+Q'$, since we've already noted that~$\Ocal_f(mP+Q')$ is
Zariski dense. This completes the proof of
Theorem~\ref{theorem:mainthm} if~$Q\in(f-1)(A)$.

We now commence the proof in the general case. The Tate
module~$T_\ell(A)$ group of~$A$ has rank~$2d$, and an isogeny is zero
if and only if it induces the trivial map on the Tate module, from
which we see that~$f$ satisfies a monic integral polynomial equation
of degree~$2d$, say
\[
  F(f) = 0 \quad\text{with $F(X)\in\ZZ[X]$ monic.}
\]
We factor~$F(X)$ as
\[
  F(X) = F_1(X)F_2(X)
\]
with
\[
  \text{$F_1(X)=(X-1)^r$, \quad $F_2(X)\in\ZZ[X]$,\quad
      and\quad $F_2(1)\ne0$.}
\] 

We first deal with the case that~$r=0$. This means
that $F(1)\ne0$. Writing $F(X)=(X-1)G(X)+F(1)$, we have
\[
  0 = F(f)Q = (f-1)G(f)Q + F(1)Q,
\]
so
\[
  F(1)Q = -(f-1)G(f)Q \in (f-1)A.
\]
Thus a non-zero multiple of~$Q$ is in~$(f-1)A$, which is the case that we 
handled earlier.

We now assume that~$r\ge1$, and we define abelian subvarieties of~$A$
by
\[
  A_1=F_1(f)A\quad\text{and}\quad A_2=F_2(f)A
\]
and consider the map
\[
  \l : A_1 \times A_2 \longrightarrow A,\quad
  \l(P_1,P_2)=P_1+P_2.
\]
Lemma~\ref{lemma:AB1B2} tells us that~$\l$ is an isogeny. More
precisely, Lemma~\ref{lemma:AB1B2}(a) says that~$\l$ is surjective,
while Lemma~\ref{lemma:AB1B2}(b) tells us that
\[
  \Ker(\l) = \bigl\{(P,-P) : P\in A_1\cap A_2\bigr\} \cong A_1\cap A_2
\]
is finite.
\par
We recall the the map $\f:A\to A$ has the form $\f(P)=f(P)+Q$ for some
fixed~$Q\in A$. The map~$\l$ is onto, so we can find a
pair
\[
  \text{$(Q_1,Q_2)\in A_1\times A_2$\quad satisfying\quad$\l(Q_1,Q_2)=Q$, i.e.,
    $Q_1+Q_2=Q$.  }
\]
We observe that $fA_1\subset A_1$ and $fA_2\subset A_2$,
since~$f$ commutes with~$F_1(f)$ and~$F_2(f)$. Writing~$f_1$ and~$f_2$
for the restrictions of~$f$ to~$A_1$ and~$A_2$, respectively, we
define maps
\begin{align*}
  \f_1:A_1 &\longrightarrow A_1, & \f_1(P_1)=f_1(P_1)+Q_1, \\*
  \f_2:A_2 &\longrightarrow A_2, & \f_2(P_2)=f_2(P_2)+Q_2.
\end{align*}
Then
\begin{align*}
  \l\circ(\f_1\times\f_2)(P_1,P_2)
  &= \l\bigl(f_1(P_1)+Q_1,f_2(P_2)+Q_2\bigr)\\*
  &= f(P_1)+Q_1 + f(P_2)+Q_2\\
  &= f(P_1+P_2) + Q\\*
  &= \f\circ\l(P_1,P_2),
\end{align*}
which shows that we have a commutative diagram
\[
  \begin{CD}
  A_1\times A_2 @>\l>> A \\
  @VV \f_1\times \f_2 V    @VV \f V \\
  A_1\times A_2 @>\l>> A \\
  \end{CD}
\]

The map~$\l$ is an isogeny, so in particular it is a finite morphism,
so Lemma~\ref{lemma:finitemapsameda} with $X=A_1\times A_2$ and $Y=A$
says that
\begin{equation}
  \label{eqn:aaa}
  \d(\f_1\times \f_2) = \d(\f)
  \quad\text{and}\quad
  \a\bigl(\f_1\times \f_2,(P_1,P_2)\bigr) = \a(\f,P_1+P_2).
\end{equation}
Next we apply Lemma~\ref{lemma:ddadproduct} with $X=A_1$ and $Y=A_2$
to conclude that
\begin{align}
  \label{eqn:bbb}
  \d(\f_1\times \f_2) &= \max\bigl\{\d(\f_1),\d(\f_2)\bigr\},\\
  \label{eqn:ccc}
  \a\bigl(\f_1\times \f_2,(P_1,P_2)\bigr)
   &= \max\bigl\{\a(\f_1,P_1),\a(\f_2,P_2)\bigr\}.
\end{align}

We now fix a point~$P\in A$ whose orbit~$\Ocal_\f(P)$ is Zariski dense
in~$A$. Since~$\l$ is onto, we can write
\begin{equation}
  \label{eqn:eee}
  P=\l(P_1,P_2)=P_1+P_2 \quad\text{for some~$P_1\in A_1$ and~$P_2\in A_2$.}
\end{equation}
Then Lemma~\ref{lemma:finitemapsameda}(a) tells us that
the~$(\f_1\times\f_2)$-orbit of~$(P_1,P_2)$ is Zariski dense
in~$A_1\times A_2$, after which Lemma~\ref{lemma:ddadproduct}(a) tells
us that~$\Ocal_{\f_1}(P_1)$ is Zariski dense in~$A_1$
and~$\Ocal_{\f_2}(P_2)$ is Zariski dense in~$A_2$.

Under the assumption that~$\overline{\Ocal_{\f_1}(P_1)}=A_1$
and~$\overline{\Ocal_{\f_2}(P_2)}=A_2$, we are going to prove the
following result.
\begin{claim}
\label{claim:af1p1f2p2}
\begin{equation}
  \label{eqn:ddd}
  \a(\f_1,P_1)=\d(\f_1)
  \quad\text{and}\quad
  \a(\f_2,P_2)=\d(\f_2),
\end{equation}
\end{claim}
Assuming this claim, the following computation completes the proof of
Theorem~\ref{theorem:mainthm}:
\begin{align*}
  \a(\f,P)
  &= \a(\f,P_1+P_2) 
    && \text{from \eqref{eqn:eee},}\\
  &= \a\bigl(\f_1\times \f_2,(P_1,P_2)\bigr) 
    && \text{from \eqref{eqn:aaa},}\\
  &= \max\bigl\{\a(\f_1,P_1),\a(\f_2,P_2)\bigr\}
    && \text{from \eqref{eqn:ccc},}\\
  &= \max\bigl\{\d(\f_1),\d(\f_2)\bigr\}
    && \text{from \eqref{eqn:ddd},}\\
  &= \d(\f_1\times\f_2)
    && \text{from \eqref{eqn:bbb},}\\
  &=\d(\f)
    && \text{from \eqref{eqn:aaa}.}
\end{align*}
\par
We now prove Claim~\ref{claim:af1p1f2p2}.  We note that if~$R\in A$
is in the kernel of the isogeny~$f-1$, then
\[
  \rho R = \bigl(G_1(f)(f-1)^r+G_2(f)F_2(f)\bigr)R
    = G_2(f)F_2(f)R \in A_2.
\]
Hence
\[
  R \in A_1 \cap \Ker(f-1)
  \enspace\Longrightarrow\enspace
  \rho R \in A_1\cap A_2 \subset A[\rho]
  \enspace\Longrightarrow\enspace
  R \in A[\rho^2].
\]
This proves that the group endomorphism
\[
  f_1-1 : A_1\longrightarrow A_1
\]
has finite kernel, so it is surjective. In particular, the
point~$Q_1\in A_1$ is in the image of~$f_1-1$, so
$\a_{\f_1}(P_1)=\d_{\f_1}$ from the special case of the theorem with
which we started the proof. This proves the first statement in
Claim~\ref{claim:af1p1f2p2}.

For the second statement in Claim~\ref{claim:af1p1f2p2}, we will show
that both $\a(\f_1,P_2)$ and~$\d(\f_2)$ are equal to~$1$. We use the
following elementary result.

\begin{lemma}
\label{lemma:XnX1r}
Fix $r\ge1$. There are polynomials~$c_{r,j}(T)\in\ZZ[T]$ of degree
at most~$r-1$ so that for all $n\ge0$ we have
\[
  X^n \equiv \sum_{j=0}^{r-1} c_{r,j}(n)X^j \pmod{(X-1)^r}.
\]
\end{lemma}
\begin{proof}
We compute
\begin{align*}
  X^n
  &= \bigl((X-1)+1\bigr)^n\\
  &= \sum_{k=0}^n \binom{n}{k} (X-1)^k \\
  &\equiv \sum_{k=0}^{r-1} \binom{n}{k} (X-1)^k \pmod{(X-1)^r}\\
  &\equiv \sum_{k=0}^{r-1} \binom{n}{k} \sum_{j=0}^k \binom{k}{j}(-1)^{k-j}X^j
        \pmod{(X-1)^r}\\
  &\equiv \sum_{j=0}^{r-1} 
     \left[\sum_{k=j}^{r-1}(-1)^{k-j} \binom{k}{j}\binom{n}{k}\right]  X^j
        \pmod{(X-1)^r}.
\end{align*}
The quantity in braces is~$c_{r,j}(n)$. 
\end{proof}

We now observe that
\[
  (f-1)^rA_2 = (f-1)^rF_2(f)A_2=F(f)A_2=0,
\]
since~$F(f)=0$, so we see that~$(f_2-1)^r$ kills~$A_2$.  So using
Lemma~\ref{lemma:XnX1r}, we find that the action of the iterates
of~$f_2$ on~$A_2$ is given by
\[
  f_2^n =  \sum_{j=0}^{r-1} c_{r,j}(n) f_2^j\in \End(A_2).
\]
Note that the polynomials~$c_{r,j}$ have degree at most~$r-1$ and do
not depend on~$n$.  

Let~$H_2$ be an ample symmetric divisor on~$A_2$, and
let~$d_2=\dim(A_2)$. Then
\begin{align*}
  \bigl((f_2^n)^*H_2\bigr)\cdot H_2^{d_2-1}  
  &= \left(\sum_{j=0}^{r-1} c_{r,j}(n) f_2^j\right)^*H_2 \cdot H_2^{d_2-1}  \\
  &= \sum_{j=0}^{r-1} \bigl(c_{r,j}(n) f_2^j\bigr)^*H_2 \cdot H_2^{d_2-1}  \\
  &= \sum_{j=0}^{r-1} c_{r,j}(n)^2 (f_2^j)^*H_2 \cdot H_2^{d_2-1}  \\
  &\le C(A_2,H_2,f_2)n^{2r-2},
\end{align*}
since the $c_{r,j}$ polynomials have degree at most~$r-1$.  (We have
also used the fact that since~$H_2$ is symmetric, we have $m^*H_2\sim
m^2H_2$ for any integer~$m$.)
This allow us to compute
\[
  \d(f_2) = \lim_{n\to\infty} 
   \Bigl(\bigl((f_2^n)^*H_2\bigr)\cdot H_2^{d_2-1}\Bigr)^{1/n}
  \le  \lim_{n\to\infty}  \bigl(C(A_2,H_2,f_2)n^{2r-2}\bigr)^{1/n}
  = 1,
\]
which shows that~$\d(f_2)=1$. 

We next do a similar height calculation. To ease notation, we write
\[
  \|R\| = \sqrt{\hhat_{A_2,H_2}(R)}
\]
for the norm associated to the $H_2$-canonical height on~$A_2$.
(See~\cite{hindrysilverman:diophantinegeometry,lang:diophantinegeometry}
for basic properties of canonical heights on abelian varieties.)  Then
\begin{align*}
  \bigl\| \f_2^n(P_2) \bigr\|
  &= \left\| f_2^n(P_2) + \sum_{i=0}^{n-1} f_2^i(Q_2) \right\| \\
  &= \left\| \sum_{j=0}^{r-1} c_{r,j}(n) f_2^j(P_2) 
      + \sum_{i=0}^{n-1} \sum_{j=0}^{r-1} c_{r,j}(i) f_2^j(Q_2) \right\| \\
  &\le (r+nr) \max_{\substack{0\le j<r\\ 0\le i\le n\\}} \bigl|c_{r,j}(i)\bigr|
       \cdot \max_{0\le j<r} \bigl\|f_2^j(Q_2)\bigr\| \\
  &\le C'(r,f_2,Q_2) n^r.
\end{align*}
This allows us to compute
\begin{align*}
  \a(\f_2,P_2)
  &=  \lim_{n\to\infty} \hhat_{H_2}\bigl(\f_2^n(P_2)\bigr)^{1/n} \\
  &\le  \lim_{n\to\infty} \bigl(C'(r,f_2,Q_2) n^r\bigr)^{2/n} \\  
  &= 1.
\end{align*}
Hence $\a(\f_2,P_2)=1$, which is also equal to~$\d(\f_2)$. This
completes the proof of the second part of Claim~\ref{claim:af1p1f2p2}
and with it, the proof of Theorem~\ref{theorem:mainthm}.
\end{proof}

\section{An Auxiliary Lemma}
\label{section:auxlemma}
In this final section we prove a lemma that is a bit stronger than is
needed to justify formula~\eqref{eqn:fxNllfYN}, which we used in the
proof of Lemma~\ref{lemma:finitemapsameda}.

\begin{lemma}
\label{lemma:lffl}
\begin{parts}
\Part{(a)}
Let $X,Y,Z$ be non-singular varieties, let $\l:Y\to Z$ be a morphism,
and let $\f:X\dashrightarrow Y$ be a rational map. Then
\[
  (\l\circ\f)^*=\f^*\circ\l^*
  \quad\text{as maps $\Pic(Z)\to\Pic(X)$.}
\]
\Part{(b)}
Let $W,X,Y$ be non-singular varieties, let $\l:W\to X$ be a finite
morphism, and let $\f:X\dashrightarrow Y$ be a rational map. Then
\[
  (\f\circ\l)^* = \l^*\circ\f^*
  \quad\text{as maps $\Pic(Y)\to\Pic(W)$.}
\]
\end{parts}
\end{lemma}
\begin{proof}
(a) We blow up $\pi:\tilde X\to X$ to resolve the map~$\f$, so we 
have a commutative diagram
\[
  \begin{array}{ccccc}
    \tilde X \\
    \Big\downarrow{\scriptstyle\pi} & {\displaystyle\searrow}\,
             \raisebox{10pt}{\hbox{$\scriptstyle\tilde\f$}} \\
    X & \stackrel{\f}{\dashrightarrow} & Y & \xrightarrow{\;\l\;} Z \\
  \end{array}
\]
where~$\pi$ is a birational map and~$\tilde\f$ is a morphism.
Let~$D\in\Pic(Z)$. The map~$\l\circ\tilde\f$ is a morphism resolving
the rational map~$\l\circ\f$, so 
\begin{align*}
  (\l\circ\f)^*D 
  &= \pi_*(\l\circ\tilde\f)^* D
    &&\text{by definition of pull-back,} \\
  &= \pi_*\circ(\tilde\f^*\circ\l)^* D
     &&\text{since $\tilde\f$ and $\l$ are morphisms,} \\
  &= (\pi_*\circ\tilde\f^*)\circ\l^* D \\
  &= \f^*\circ\l^* D
    &&\text{by definition of pull-back.}
\end{align*}
\par\noindent(b)\enspace
We blow up $\pi:\tilde X\to X$ to resolve the map~$\f$, and then we blow
up~$W$ to resolve the map $\pi^{-1}\circ\l$. This gives a commutative
diagram
\[
  \begin{array}{ccccc}
    \tilde W & \xrightarrow{\;\tilde\l\;} & \tilde X \\
    \Big\downarrow{\scriptstyle\mu} & &
    \Big\downarrow{\scriptstyle\pi} & {\displaystyle\searrow}\,
             \raisebox{10pt}{\hbox{$\scriptstyle\tilde\f$}} \\
    W & \xrightarrow{\;\l\;} &  X & \stackrel{\f}{\dashrightarrow} & Y \\
  \end{array}
\]
Here~$\mu$ and~$\pi$ are birational morphisms and~$\tilde\l$
and~$\tilde\f$ are morphisms. We claim that
\begin{equation}
  \label{eqn:lphimul}
  \l^*\circ\pi_* = \mu_*\circ\tilde\l^*.
\end{equation}
Assuming the validity of~\eqref{eqn:lphimul}, we compute
\begin{align*}
  \l^*\circ\f^* D
  &= \l^*\circ\pi_*\circ\tilde\f^*D
    &&\text{by definition of pull-back,} \\
  &= \mu_*\circ\tilde\l^*\circ\tilde\f^*D
      &&\text{from \eqref{eqn:lphimul},} \\
  &= \mu_*\circ(\tilde\f\circ\tilde\l)^*D
      &&\text{since $\tilde\f$ and $\tilde\l$ are morphisms,} \\
  &= (\f\circ\l)^*D
    &&\text{by definition of pull-back,} 
\end{align*}
where for the last line we have used the fact that
$\tilde\f\circ\tilde\l$ is a morphism the resolves the rational
map~$\f\circ\l$. It remains to verify~\eqref{eqn:lphimul}.\footnote{We
  remark that~\eqref{eqn:lphimul} requires~$\l$ be a finite map. It is
  not true for morphisms, even birational morphisms. For example, let
  $W=\tilde X$ and $\l=\pi$ and $\mu=\text{id}_W$, then
  $\l^*\circ\pi_*=\pi^*\circ\pi_*$ kills exceptional divisors, while
  $\mu_*\circ\tilde\l^*$ is the identity map.}

Let $D\in\Div(\tilde X)$ be an irreducible divisor, and let~$|D|$
denote the support of~$D$. There are two cases.  First suppose
that~$D$ is an exceptional divisor, so $\pi_*D=0$. This means that
$\dim \pi\bigl(|D|\bigr)\le\dim(X)-2$, and since~$\tilde\l$ is surjective,
we have $\pi\bigl(|D|\bigr)=\l\circ\mu\circ\tilde\l^{-1}\bigl(|D|\bigr)$.
Hence
\[
  \dim \l\circ\mu\circ\tilde\l^{-1}\bigl(|D|\bigr) \le \dim(X)-2.
\]
We now use the fact that~$\l$ is a finite map to deduce that
\[
  \dim \mu\circ\tilde\l^{-1}\bigl(|D|\bigr) \le \dim(W)-2.
\]
It follows that
\[
  \mu_*\circ\tilde\l^*D = 0.
\]
Next suppose that~$D$ is a horizontal divisor relative to~$\pi$, so
$D=\pi^*\circ\pi_*D$. This allows us to compute
\begin{align*}
  \mu_*\circ\tilde\l^*D
  &= \mu_*\circ\tilde\l^*\circ\pi^*\circ\pi_*D
  &&\text{using $D=\pi^*\circ\pi_*D$,} \\
  &= \mu_*\circ(\pi\circ\tilde\l)^*\circ\pi_*D
  &&\text{since $\tilde\l$ and~$\pi$ are morphisms,}\\
  &= \mu_*\circ(\l\circ\mu)^*\circ\pi_*D
  &&\text{commutativity of the diagram} \\
  &= \mu_*\circ\mu^*\circ\l^*\circ\pi_*D
  &&\text{since $\l$ and~$\mu$ are morphisms,}\\
  &= \l^*\circ\pi_*D 
  &&\text{since $\mu_*\circ\mu^*=\text{id}_W^*$.}
\end{align*}
This shows in both cases that~$\mu_*\circ\tilde\l^*=\l^*\circ\pi_*$, which
completes the proof of~\eqref{eqn:lphimul}.
\end{proof}

\begin{acknowledgement}
I would like to thank Shu Kawaguchi for his helpful suggestions.
\end{acknowledgement}


\end{document}